\documentclass[leqno,12pt]{amsart}
\usepackage{latexsym}
\usepackage{mathrsfs}
\usepackage{amsmath}
\usepackage{amssymb}
\usepackage[bookmarks]{hyperref}
\usepackage{pstricks,pst-plot}

\newtheorem{theorem}{Theorem}[section]
\newtheorem{lemma}[theorem]{Lemma}
\newtheorem{corollary}[theorem]{Corollary}
\newtheorem{proposition}[theorem]{Proposition}

\theoremstyle{definition}

\newtheorem{definition}[theorem]{Definition}

\newtheorem{question}[theorem]{Question}

\font\tenscr=rsfs10  scaled \magstep1

\font\sevenscr=rsfs7  scaled \magstep1
\font\fivescr=rsfs5  scaled \magstep1
\skewchar\tenscr='177 \skewchar\sevenscr='177 \skewchar\fivescr='177
\newfam\scrfam \textfont\scrfam=\tenscr \scriptfont\scrfam=\sevenscr
\scriptscriptfont\scrfam=\fivescr

\newcommand{\C}{\mathbb{C}}

\newcommand{\D}{\mathbb{D}}

\newcommand{\Z}{\mathbb{Z}}

\newcommand{\R}{\mathbb{R}}

\newcommand{\bthm}{\begin{theorem}}
\newcommand{\ethm}{\end{theorem}}
\newcommand{\blem}{\begin{lemma}}
\newcommand{\elem}{\end{lemma}}
\newcommand{\bcor}{\begin{corollary}}
\newcommand{\ecor}{\end{corollary}}
\newcommand{\bprop}{\begin{proposition}}
\newcommand{\eprop}{\end{proposition}}
\newcommand{\bdefn}{\begin{definition}}
\newcommand{\edefn}{\end{definition}}
\newcommand{\bpf}{\begin{proof}}
\newcommand{\epf}{\end{proof}}
\newcommand{\bques}{\begin{question}}
\newcommand{\eques}{\end{question}}

\def\vep {\varepsilon}
\def \sm {\setminus}

\def\tA{\widetilde A}

\def\itemskip {\vskip 3pt plus 2 pt minus 1 pt}
\def \ma {\mathfrak{M}_A}

\def\sf{{\mathscr F}}

\newcommand{\ra}{\rightarrow}

\newcommand{\ol}{\overline}

\def\row#1#2{#1_1,\ldots, #1_{#2}}
\def\vector#1#2{(#1_1,\ldots, #1_{#2})}

\def\area{B}

\newcommand{\mcal}{\mathcal}
\def\pD{\partial \D}
\def\sf{{\mathscr F}}
\def\logai{\log|A^{-1}|}
\def\sequence#1#2{(#1_{#2})_{#2=1}^\infty}
\def\sequencecurly#1#2{\{#1_{#2}\}_{#2=1}^\infty}
\def\doubleindex#1#2#3{\{#1_{#2,#3}\}_{#3=1}^\infty}
\def\Ao{A_\omega}
\def\Xo{X_\omega}
\def\Sigmao{\Sigma_\omega}
\def\Co{\C^\omega}
\def\sO{\ol{ \mathscr O}}
\def\crx#1{C_\R(#1)}

\def \mao {\mathfrak{M}_{\Ao}} 
\def\tX{\widetilde X}
\def\eik{{E_I^k}}
\def\eck{{E_J^k}}

\begin{document}
\title[Dirichlet on its maximal ideal space]{A nontrivial uniform algebra\\ Dirichlet on its maximal ideal space}

\author{Alexander J. Izzo}
\address{Department of Mathematics and Statistics, Bowling Green State University, Bowling Green, OH 43403}
\email{aizzo@bgsu.edu}
\thanks{The author was partially supported by NSF Grant DMS-1856010.}


\begin{abstract}
It is shown$\vphantom{\widehat{\widehat{\widehat{\widehat{\widehat{\widehat{\widehat X}}}}}}}$ 
that there exists a nontrivial uniform algebra that is Dirichlet on its maximal ideal space and has a dense set of elements that are exponentials.  This answers a 65-year-old question of John~Wermer and a 17-year-old question of Garth~Dales and Joel~Feinstein.
Our example is $P(X)$ for a certain compact set $X$ in $\C^2$.  It is also shown that there exists a logmodular uniform algebra with proper Shilov boundary but with no nontrivial Gleason parts.  This answers a modification of another 65-year-old question of Wermer.
\end{abstract}

\maketitle

\vskip -2.21 true in
\centerline{\footnotesize\it Dedicated to the memory of John Wermer} 
\vskip 2.21 truein

%
%
%
%

\section{Introduction}

In the 1960 paper \cite{W0} in which John Wermer proved his embedding theorem, he raised five questions about Dirichlet algebras.  The last three of these were solved within the next eight years, but the first two have remained open.  The first of these questions also appears in the problem list from the Function Algebras conference held in Tulane in 1965 \cite[p.~348, Question~25]{Birtel} and in Andrew Browder's classic book on uniform algebras~\cite[p.~208]{Browder}.  In this paper, we answer the first question and answer a modification of the second.  We also answer a question raised by Garth Dales and Joel Feinstein \cite{DF} in 2008. 

We call a uniform algebra $A$ on a compact space $X$ \emph{trivial} if $A=C(X)$ and \emph{nontrivial} otherwise.  We call a Gleason part \emph{trivial} if it consists of a single point and \emph{nontrivial} otherwise. 
(For definitions of other terminology and notation used in this introduction see Section~\ref{notation}.)

The questions at issue are as follows.

\bques\cite[p.~381~(a)]{W0}\label{WQ1}
Let $A$ be a nontrivial uniform algebra that is Dirichlet on a compact space $X$.  Must $X$ be a proper subset of the maximal ideal space of $A$?
\eques

\bques\cite[p.~381~(b))]{W0}\label{WQ2}
Let $A$ be a Dirichlet algebra on a compact space $X$.  If $X$ is a proper subset of the maximal ideal space of $A$, can every Gleason part for $A$ be trivial?
\eques

\bques\cite[p.~1302 (2)]{DF}\label{DFQ}
Let $A$ be a uniform algebra.  If the set $\exp A = \{\exp a: a\in A\}$ of exponentials in $A$ is dense in $A$ must $A$ be trivial?
\eques

We will prove the following theorems answering Questions~\ref{WQ1} and~\ref{DFQ} and answering the modification of Question~\ref{WQ2} obtained by replacing \emph{Dirichlet} by \emph{logmodular}.

\bthm\label{main1}
There exists a nontrivial uniform algebra $A$ with maximal ideal space $X$
such that $A$ is Dirichlet on $X$.
\ethm

\bthm\label{main2}
There exists a logmodular algebra $A$ on a compact space $X$
such that $X$ is a proper subset of the maximal ideal space of $A$ but every Gleason part for $A$ is trivial.
\ethm

\bthm\label{main3}
There exists a nontrivial uniform algebra $A$
such that the set $\exp A = \{\exp a: a\in A\}$ is dense in $A$.
\ethm

We will give an example that simultaneously establishes both Theorems~\ref{main1} and~\ref{main3}.  Our example is not just an abstract uniform algebra; it is the algebra $P(X)$ of uniform limits of polynomials on some compact set $X$ in $\C^2$ of topological dimension~$1$.  Our example has the additional property of being strongly regular.  Thus we will prove the following result which contains Theorems~\ref{main1} and~\ref{main3}.

\bthm\label{Dirichlet-PX}
There exists a compact polynomially convex set $X$ in $\C^2$ of topological dimension~$1$ such that $P(X)$ is a nontrivial strongly regular uniform algebra that is Dirichlet on $X$ and has dense exponentials.
\ethm

It follows from the Arens-Royden theorem~\cite[Theorem~10.1]{S1} that the maximal ideal space of a commutative Banach algebra with dense exponentials is simply coconnected.  (We call a space $X$ \emph{simply coconnected} if its first \v Cech cohomology group $\check H^1(X;\Z)$ vanishes.) Thus the space $X$ in the above theorem is simply coconnected.  That this space $X$ also satisfies $\check H^n(X;\Z)=0$ for all $n\geq 2$ is immediate from \cite[p.~137]{HW} since the topological dimension of $X$ is $1$ (and is also immediate from \cite[Theorem~10.6]{S1} since $X$ is polynomially convex in $\C^2$).  By replacing $X$ by a nontrivial maximal set of antisymmetry for $P(X)$, we can take $X$ to be connected and hence to have trivial \v Cech cohomology in all degrees.

We will also prove more than is stated in Theorem~\ref{main2}.

\bthm\label{logmodular2}
There exists a logmodular algebra $A$, on a compact metrizable space, that has proper Shilov boundary and is such that the set $\{a^2: a\in A\}$ is dense in $A$.
\ethm

Since it is well known that there are no nontrivial Gleason parts for a uniform algebra in which the set of squares is dense (see~\cite[Lemma~1.1]{Cole} or \cite[Section~19]{S1}), Theorem~\ref{main2} is an immediate consequence of the above result.

Some historical remarks to put these theorems and the questions they answer in context are in order.  It was once a conjecture of considerable interest whether, given a uniform algebra $A$ on a compact space $X$ such that the maximal ideal space $\ma$ of $A$ is strictly larger than $X$, the set $\ma\sm X$ must contain an analytic disc.  
Andrew~Gleason~\cite{Gl} introduced his notion of parts with the hope that they would yield analytic discs in maximal ideal spaces.  (It is a consequence of the Schwarz lemma that every analytic disc in the maximal ideal space $\ma$ of a uniform algebra $A$ is a subset of some Gleason part for $A$.)  In the same paper, he also introduced the notion of a Dirichlet algebra and wrote \lq\lq It appears that this class of algebras is of considerable importance and is amenable to analysis.\rq\rq\  Two years later, Wermer proved his embedding theorem which asserts that every nontrivial Gleason part in the maximal ideal space of a Dirichlet algebra is an analytic disc~\cite{W0}.  The result left open, however, the possibility that in a Dirichlet algebra every Gleason part might consist of a single point, and thus that the maximal ideal space might contain no analytic discs.

Three years after the appearance of Wermer's embedding theorem~\cite{W0}, Gabriel~Stolzenberg \cite{Stol} disproved the conjecture about the existence of analytic discs in the maximal ideal space of a \emph{general uniform algebra} by producing a compact set $X$ in $\C^2$ such that the polynomial hull $\widehat X$ of $X$ is strictly larger than $X$ but $\widehat X\sm X$ contains no analytic discs.  Stolzenberg's result does not, however, address the question of the existence of analytic discs in maximal ideal spaces of \emph{Dirichlet} algebras.

Wermer's result that every nontrivial Gleason part for a Dirichlet algebra is an analytic disc was generalized to logmodular algebras by 
Kenneth~Hoffman~\cite{Hoffman} and then to uniform algebras with uniqueness of representing measures by 
Gunter~Lumer~\cite{Lumer}.

Theorem~\ref{main1} answers in the negative the question of whether the maximal ideal space of every nontrivial Dirichlet algebra must contain an analytic disc.  It remains open, though, whether Dirichlicity implies the existence of analytic discs when the maximal ideal space is known to be strictly larger than the Shilov boundary.
However, Theorem~\ref{main2} shows that the weaker condition of logmodularity, and hence the still weaker condition of uniqueness of representing measures, which is sufficient to insure that every nontrivial Gleason part is an analytic disc, does not imply the existence of analytic discs in these maximal ideal spaces.

By a theorem of Stuart Sidney~\cite[Theorem~3]{Sidney}, if a point having a unique representing measure lies in a trivial Gleason part, then there are no nonzero bounded point derivations at that point.  Therefore, in view of Lumer's generalization of Wermer's embedding theorem, for a uniform algebra with uniqueness of representing measures, the conditions that (i) the maximal ideal space contains no analytic discs, (ii) the uniform algebra has no nontrivial Gleason parts, and (iii) there are no nonzero bounded point derivations on the uniform algebra, are all equivalent.  In particular, there are no nonzero bounded point derivations on a logmodular algebra satisfying the conditions in Theorem~\ref{main2}.  Actually, for the particular logmodular algebra we construct, the nonexistence of nonzero bounded point derivations is clear because, as is well known, there are no nonzero bounded point derivations on a uniform algebra with dense squares \cite[Lemma~1.1]{Cole}.

The question of Dales and Feinstein about uniform algebras with dense exponentials was also motivated by the consideration of analytic structure  in maximal ideal spaces.  The condition that a uniform algebra has dense invertible group is strictly stronger than the condition that its maximal ideal space contains no analytic discs.  This suggests the conjecture, raised as a question by Thomas~Dawson and Feinstein~\cite{DawsonF} that a uniform algebra with proper Shilov boundary can never have dense invertible group.
Dales and Feinstein \cite[Theorem~2.1]{DF} gave a counterexample.  Their result led them to ask whether a uniform algebra with proper Shilov boundary can have dense \emph{exponentials}, and in turn, to ask whether there even exists a nontrivial uniform algebra with dense exponentials.

Theorem~\ref{main1} should be compared with
Lavrentiev's theorem 
\cite[Theorem~2.11]{AW} which asserts that if $K$ is a compact planar set with empty interior and connected complement, then $P(K)=C(K)$.  The hypotheses imply that the maximal ideal space of $P(K)$ is $K$, and a standard proof of Lavrentiev's theorem consists of proving that under the hypotheses $P(K)$ is Dirichlet on $K$, and then using the Dirichlicity to show that $P(K)=C(K)$.  Thus the statement that every uniform algebra Dirichlet on its maximal ideal space is trivial, if it were true, could be regarded as a generalization of Lavrentiev's theorem.  Theorem~\ref{main1} shows that Lavrentiev's theorem does not generalize in this manner.  For another perspective on the relationship of our results to Lavrentiev's theorem, note that a compact planar set $K$ has empty interior and connected complement if and only if the topological dimension $\dim K$ of $K$ is at most $1$ and $K$ is simply coconnected.  (That a planar set has topological dimension at most $1$ if and only if its interior is empty is a standard result in dimension theory \cite[Theorem~IV~3]{HW}.
That a compact planar set $K$ has connected complement if and only if $\check H^1(K;\Z)=0$ is an immediate consequence of Alexander-Pontryagin duality.  A more elementary proof is given in \cite{Browdermonthly}.)  Theorem~\ref{Dirichlet-PX} (and the remark that immediately follows it) shows that these two topological conditions on a compact metrizable space $K$ are \emph{not} sufficient to imply that every uniform algebra with maximal ideal space $K$ is trivial.  However, not every compact metrizable,
simply coconnected space of topological dimension~$1$ 
embeds in the plane; we give an example at the end of the paper.  These observations suggest the following open question.

\begin{question}
Let $K$ be a compact planar set with empty interior and connected complement.  Is every uniform algebra with maximal ideal space $K$ trivial?
\end{question}

This question is of course related to the 68-year-old question of I.~M.~Gelfand \cite{Gelfand} of
whether every uniform algebra with maximal ideal space the closed unit interval $[0,1]$ is trivial.  In connection with this, it should be noted that the uniform algebra in Theorem~\ref{Dirichlet-PX} is the first known example of a nontrivial uniform algebra whose maximal ideal space is both simply coconnected and of topological dimension~$1$.  However, unlike the closed unit interval, the space $X$ in the theorem fails to be locally connected, since a uniform algebra with dense exponentials necessarily has dense squares, and hence, if defined on a locally connected space, must be trivial by a result of \v Cirka \cite[Theorem~13.15 and Lemma~13.16]{S1}.  Thus Theorem~\ref{Dirichlet-PX} leaves open the possibility that every uniform algebra whose maximal ideal space is locally connected, simply coconnected, and of topological dimension~$1$ is trivial.

It is well known that if a uniform algebra $A$ is Dirichlet on a compact metrizable space $X$, then every point of $X$ is a peak point for $A$.  Thus Theorem~\ref{Dirichlet-PX} strengthens the celebrated theorem of Brian~Cole
that there exist nontrivial uniform algebras such that every point of the maximal ideal space is a peak point~\cite{Cole}
(see also~\cite[Section~19]{S1} and~\cite[Appendix]{Browder}).
Indeed Cole's construction was a source of inspiration for our proofs.

It was conjectured in the 1960s that no nontrivial strongly regular uniform algebras exist.  Donald Wilken~\cite{Wilken2} proved that there are no nontrivial strongly regular uniform algebras on the closed unit interval.  Subsequently, Donald Chalice~\cite{Chalice} proved that the same holds for uniform algebras on the unit circle.   The first example of a nontrivial strongly regular uniform algebra was given by Fein\-stein~\cite{F1} in 1990.  Theorem~\ref{Dirichlet-PX} shows that Wilken's theorem does not generalize from the closed unit interval to simply coconnected spaces of topological dimension~$1$.

A regular uniform algebra is never contained in a maximal uniform algebra \cite[p.~190]{Hoffmanbook}.  Thus Theorem~\ref{Dirichlet-PX} gives what seems to be the first known example of a Dirichlet algebra contained in no maximal uniform algebra.  The standard examples of Dirichlet algebras are, in fact, maximal, although there is a Dirichlet algebra found by Browder and Wermer \cite{BW} (see also \cite[p.~232-235]{Browder}) that is a proper subalgebra of the disc algebra and hence not maximal.  The question of whether there exists a maximal uniform algebra whose maximal ideal space coincides with its Shilov boundary appears in the problem list from the 1965 Function Algebras 
conference~\cite[p.~348, Question~23]{Birtel}.  It is easy to show that if $A$ is a  uniform algebra such that every point of the maximal ideal space of $A$ is a peak point, and $B$ is a uniform algebra that contains $A$, then the maximal ideal space of $B$ coincides with the maximal ideal space of $A$ (see \cite[Lemma~2.2]{Izzo1995}).  It follows that if there is a uniform algebra that is Dirichlet on its maximal ideal space and is contained in a maximal uniform algebra $B$, then $B$ gives an affirmative answer to the question just mentioned.

Being a compact metrizable space of topological dimension~$1$, the space $X$ of Theorem~\ref{Dirichlet-PX} embeds in $\R^3$ 
\cite[Theorem~V~2]{HW}.  It follows easily that on every compact manifold $M$ (with or without boundary) of dimension greater than or equal to $3$, there exists a nontrivial uniform algebra with maximal ideal space $M$ that is Dirichlet on $M$.  Using a construction of Feinstein and the author \cite{FI}, we will show that such a uniform algebra can in addition be essential.  

\bthm\label{three-ball}
On every compact manifold $M$ (with or without boundary) of dimension $\geq$
$3$, there exists an essential, strongly regular uniform algebra with maximal ideal space $M$ that is Dirichlet on $M$.
\ethm

Before concluding this introduction we make some remarks on the proofs of our main theorems.  The uniform algebra constructed in the proof of Theorem~\ref{logmodular2} is a Cole root extension of the disc algebra.  However, to obtain a uniform algebra that we can prove is logmodular, we must carry out the construction in a special way.  

Usually to obtain a uniform algebra with dense squares on a compact metrizable space via a Cole root extension one proceeds as follows~\cite{Cole} (or see~\cite[Section~19]{S1}).  
Let $A_0$ be a uniform algebra on a compact metrizable space $X_0$, and let $\sf_0$ be a countable dense subset of $A_0$.  Give $\C^{\sf_0}$ the product topology.  Let $p_1:X_0\times \C^{\sf_0}\ra X_0$ and $p_f:X_0\times \C^{\sf_0}\ra\C$ denote the projections given by $p_1(x, (z_g)_{g\in \sf_0})=x$ and $p_f(x, (z_g)_{g\in\sf_0})=z_f$.  Define $X_1\subset X_0\times \C^{\sf_0}$ by
\[
(*)\qquad  X_1 = \{\, y\in X_0\times \C^{\sf_0}:\bigl(p_f(y)\bigr)^2=f\bigl(p_1(y)\bigr) \ \hbox{for all $f\in\sf_0$}\,\},\phantom{(*)}
\]
and let $A_1$ be the uniform algebra on $X_1$ generated by the set of functions $\{\, f\circ p_1: f\in A_0\} \cup \{\, p_f: f\in\sf_0\}$.    The map $\pi_1=p_1|X_1$ is surjective, and there is an isometric embedding $\pi_1^*:A_0\ra A_1$ given by $\pi_1^*(f)=f\circ \pi_1$, so we can regard $A_0$ as a subalgebra of $A_1$.  Since $(p_f|X_1)^2=f\circ\pi_1$, then each function $f\in\sf_0$ has a square root in $A_1$.  One then iterates this construction to obtain a direct system $\{(A_m, X_m, \pi_m)\}_{m=0}^\infty$, and one takes the limit of this direct system to obtain the desired final uniform algebra $A_\omega$.

To prove Theorem~\ref{logmodular2}, one would like to show that if one applies the above construction taking $A_0$ to be the disc algebra, then $A_\omega$ is logmodular.  However, our proof of logmodularity will use results about uniform algebras of limits of holomorphic functions on a Riemann surface and consequently will require that in the direct system $\{(A_m, X_m, \pi_m)\}_{m=0}^\infty$ each of the $A_m$ is (essentially) such a uniform algebra.  This precludes us from taking the collection $\sf_0$ (and the corresponding collections as the iteration proceeds) to be infinite, and hence precludes $\sf_0$ from being dense.  Therefore, instead, at each step of the iteration we create a square root for only a single function.  We must then choose these functions in such a way that there is a strictly increasing  infinite sequence 
of positive integers $(n_k)_{k=1}^\infty$
such that for each $n_k$ there is a countable dense set of functions in $A_{n_k}$ each of which is eventually given a square root.

The construction in the proof of Theorem~\ref{Dirichlet-PX} also involves a sequence of uniform algebras on spaces that are (essentially) Riemann surfaces.  However, since in this case we are aiming for dense exponentials rather than dense squares, we can no longer use root extensions.  One would like to simply replace, in equation $(*)$, the equality $\bigl(p_f(y)\bigr)^2=f\bigl(p_1(y)\bigr)$ by $\exp\bigl(p_f(y)\bigr) = f\bigl(p_1(y)\bigr)$, but of course the resulting space $X_1$ would then not be compact.  To obtain a compact space we cut this set off by requiring that the imaginary part of $p_f(y)$ lie in some chosen closed, bounded interval.  One of the key issues in the proof is to show that by choosing this interval large enough we can make the resulting uniform algebra nontrivial, and furthermore, insure that nontriviality persists when we take the direct limit to obtain the final uniform algebra.  In the Cole construction, nontriviality is proven using an averaging operator, an object which we do \emph{not} have in the construction we use to prove Theorem~\ref{Dirichlet-PX}.

In the next section we present terminology and notation some of which has already been used above.   Section~\ref{lemmas} consists of lemmas we will need.  Sections~\ref{logmodular2-section}, \ref{Dirichlet-PX-section}, and \ref{three-ball-section} are devoted to the proofs of Theorems~\ref{logmodular2}, \ref{Dirichlet-PX}, and \ref{three-ball}, respectively.  As already mentioned, in the last section we give an example of a compact metrizable,
simply coconnected space of topological dimension~$1$ that does not embed in the plane.

The author thanks Lee Stout and Paul Gauthier for reading a draft of the paper and making comments that led to expository improvements.  An especially big thank you is due to Anthony O'Farrell who checked every detail of the proofs and provided many helpful comments.

It is with a mixture of joy and sorrow that I dedicate this paper to the memory of John Wermer.  Sorrow, of course, that he is no longer with us; joy that I had the privilege of knowing and working with him.

%
%
%
%

\section{Terminology and Notation}\label{notation}

We give here some terminology and notation we will use.  Readers well versed in uniform algebra concepts can safely skip most of this but should probably read at least the next two paragraphs and especially the last three paragraphs.

We denote the real and imaginary parts of a complex number or function $z$ by $\Re z$ and $\Im z$, respectively.
We denote the set of positive integers by $\Z_+$.

Throughout the paper the word \emph{space} will mean \emph{Hausdorff space}.  In fact, all the spaces that arise in this paper are metrizable.  
By the topological dimension of a space we mean the usual Lebesgue covering dimension.

Let $X$ be a compact space.  We denote by $\crx X$ the (real) algebra of all continuous real-valued functions on $X$ and by
$C(X)$ the (complex) algebra of all continuous complex-valued functions on $X$, each with the supremum norm
$ \|f\| = \|f\|_X = \sup\{ |f(x)| : x \in X \}$.  A \emph{uniform algebra} $A$ on $X$ is a closed subalgebra of $C(X)$ that contains the constant functions and separates
the points of $X$.  

For $A$ a uniform algebra on a compact space $X$, a point $x$ of $X$ is said to be a \emph{peak point} for $A$ if there is a function $f$ in $A$ such that $f(x)=1$ and $|f(y)|<1$ for every $y\in X\sm \{x\}$.
The \emph{Shilov boundary} for $A$ is the smallest closed subset $\Gamma$ of $X$ such that every function in $A$ takes its maximum modulus on $\Gamma$.

The uniform algebra $A$ is said to be \emph{Dirichlet on $X$} if the set 
$\Re(A)=\{\Re f : f \in A\}$ of real parts of the functions in $A$ is dense in $\crx X$.  
It is well known that
if $A$ is Dirichlet on a compact metrizable space $X$, then every point of $X$ is a peak point for $A$, so $X$ is the Shilov boundary for $A$.  Thus if a uniform algebra $A$ is Dirichlet on its maximal ideal space, then the Shilov boundary for $A$ and the maximal ideal space of $A$ coincide.  A uniform algebra that is Dirichlet on its Shilov boundary is referred to simply as a \emph{Dirichlet algebra}.

We denote the set of invertible elements of $A$ by $A^{-1}$.  The uniform algebra $A$ is \emph{logmodular on $X$} if the set $\log|A^{-1}| = \{ \log |f| : f\in A^{-1}\}$ is dense in $\crx X$.   The statements in the previous paragraph about peak points and the Shilov boundary for Dirichlet algebras hold also for logmodular algebras.  A uniform algebra that is logmodular on its Shilov boundary is referred to simply as a \emph{logmodular algebra}.
Dirichlet algebras are logmodular. 

A uniform algebra $A$ on $X$ is said to be 
\itemskip
\begin{enumerate}
\item[(a)] \emph{natural} if the maximal ideal space of $A$ is $X$ (under the usual identification of a point of $X$ with the corresponding multiplicative linear functional),
\itemskip
\item[(b)] \emph{essential} if there is no proper closed subset $E$ of $X$ such that $A$ contains every continuous complex-valued function on $X$ that vanishes on $E$,
\itemskip
\item[(d)] \emph{regular on $X$} if for each closed set $K_0$ of $X$ and each point $x$ of $X\setminus K_0$, there exists a function $f$ in $A$ such that $f(x)=1$ and $f=0$ on $K_0$,
\itemskip
\item[(e)] \emph{normal on $X$} if for each pair of disjoint closed sets $K_0$ and $K_1$ of $X$, there exists a function $f$ in $A$ such that $f=1$ on $K_1$ and $f=0$ on $K_0$,
\end{enumerate}
The uniform algebra $A$ on $X$ is \emph{regular} or \emph{normal}, if $A$ is natural and is regular on $X$ or normal on $X$, respectively.
In fact, every regular uniform algebra is normal \cite[Theorem~27.2]{S1}.
Also, if a uniform algebra $A$ is normal on $X$, then $A$ is necessarily natural \cite[Theorem~27.3]{S1}.

Let $A$ be a uniform algebra on $X$, and let $x\in X$.  We define the ideals $M_x(A)$ and $J_x(A)$ by
\begin{align} 
M_x(A) &= \{\, f\in A:  f(x)=0\,\}  \nonumber  \\ 
\intertext{and} 
J_x(A)&= \{ \, f\in A: f^{-1}(0)\ \hbox{contains a neighborhood of $x$ in $X$}\}. \nonumber 
\end{align}
The uniform algebra $A$ is \emph{strongly regular at} $x$ if $\ol{J_x(A)}=M_x(A)$, and $A$ is \emph{strongly regular} if $A$ is strongly regular at every point of $X$.  It was shown by Wilken that every strongly regular uniform algebra is normal \cite[Corollary~1]{Wilken2}.

The \emph{Gleason parts} for the uniform algebra $A$ are the equivalence classes in the maximal ideal space of $A$ under the equivalence relation $\varphi\sim\psi$ if $\|\varphi-\psi\|<2$ in the norm on the dual space $A^*$.  (That this really is an equivalence relation is well known but {\it not\/} obvious.)

Given a point $\varphi$ in the maximal ideal space of the uniform algebra $A$ on $X$, an \emph{Arens-Singer measure} for $\varphi$ is a positive measure $\mu$ on $X$ such that
$$\log |\varphi(f)| = \int \log |f| \, d\mu$$
for every $f\in A^{-1}$.
Every Arens-Singer measure for $\varphi$ is a representing measure for $\varphi$.

For $X$ a compact subset of $\C^N$, $1\leq N\leq \omega$,
we denote by 
$P(X)$ the uniform closure on $X\subset\C^N$ of the polynomials in the complex coordinate functions $z_1, z_2, \ldots$, and we denote by $R(X)$ the uniform closure on $X$ of the rational functions  holomorphic on (a neighborhood of) $X$.  Just as in a finite-dimensional complex Euclidean space, a polynomial on the countably infinite cartesian product $\Co$ is a finite sum of finite products of coordinate functions, and a rational function is a quotient of two polynomials.

We will make use of (purely) one-dimensional complex submanifolds of $\C^N$ without boundary.  We will refer to such a space as a \emph{possibly disconnected open Riemann surface}.

Given a possibly disconnected open Riemann surface $\Sigma$ and a compact subset $K$ of $\Sigma$, we denote by $\sO(K)$ the algebra of complex-valued functions on $K$ that can be approximated uniformly by functions holomorphic on a neighborhood of $K$ in $\Sigma$, we denote by $H(K)$ the set of real-valued functions on $K$ that extend to be harmonic on a neighborhood of $K$ in $\Sigma$, and we denote the uniform closure of $H(K)$ in $\crx K$ by $\ol H(K)$.

Given $1\leq N\leq M\leq \omega$ and a function $f$ defined on a subset $X$ of $\C^N$, it will often be convenient to regard $f$ as a function on a subset $Y$ of $\C^M$ lying over $X$ under the the canonical projection $\C^M\ra\C^N$ by identifying $f$ with the function obtained by precomposing with this projection.  This will be freely done below, often without comment.  In particular, this will allow us to regard certain algebras of functions on $X$ as subalgebras of certain other algebras of functions on $Y$.

%
%
%
%

\section{Lemmas}\label{lemmas}

\blem\label{density}
Let $\Sigma$ be a possibly disconnected open Riemann surface, let $K$ be a compact subset of $\Sigma$, and let $\partial K$ be the boundary of $K$ in $\Sigma$.  Suppose that every point of $\partial K$ has a unique Arens-Singer measure for $\sO(K)$.  Then the real-linear span of $\log|\sO(K)^{-1}|$, restricted to $\partial K$, is dense in $\crx{\partial K}$
\elem

For the proof of Lemma~\ref{density} we recall some material from the paper \cite{GR} of Theodore~Gamelin and Hugo~Rossi.  There only (\emph{connected}) open Riemann surfaces are considered, but the case of possibly disconnected open Riemann surfaces $\Sigma$ follows immediately from the connected case.  Given a compact set $K$ on $\Sigma$, and a point $p$ in $K$, there is a certain (unique) measure $\nu_p$ associated with $p$ called the \emph{Keldysh measure for} $p$.  It is a positive measure on $\partial K$ that represents $p$ on $H(K)$, i.e., $\int u\, d\nu_p = u(p)$ for every $u\in H(K)$.  

\blem\cite[Lemma~3.1]{GR}\label{GR1}
Let $K$ be a compact subset of a possibly disconnected open Riemann surface.  A Borel measure $\mu$ on $\partial K$ is an Arens-Singer measure for a point $p$ of $K$ on $\sO(K)$ if and only if it represents $p$ on $H(K)$.
\elem

\blem\cite[Corollary~6.3]{GR}\label{GR2}
Let $K$ be a compact subset of a possibly disconnected open Riemann surface.  Let $P$ denote the set of points $p\in K$ such that the Keldysh measure $\nu_p$ is the unit point mass at $p$.  If $\lambda$ is a measure on $P$ that annihilates $H(K)$, then $\lambda=0$.
\elem

The set $P$ in the preceding lemma is, in fact, the Choquet boundary for $\ol H(K)$ \cite[Corollary~6.4]{GR}, and hence is a $G_\delta$-set,

The next lemma is immediate from the proof of \cite[Lemma~1]{W1}.

\blem\label{W}
Let $K$ be a compact subset of a possibly disconnected open Riemann surface.  Every function in $H(K)$ is in the real-linear span of $\log|\sO(K)^{-1}|$.
\elem

\bpf[Proof of Lemma~\ref{density}]
Let $p\in \partial K$ be arbitrary.  By Lemma~\ref{GR1}, the Keldysh measure $\nu_p$ for $p$ is an Arens-Singer measure for $p$ on $\sO(K)$.  Thus, by our hypothesis, $\nu_p$ is the unit point mass at $p$.  Thus $\partial K$ is contained in the set $P$ of Lemma~\ref{GR2}.  Therefore, if $\lambda$ is a measure on $\partial K$ that annihilates $H(K)$, then $\lambda=0$ by Lemma~\ref{GR2}.  Consequently, $H(K)$, restricted to $\partial K$, is dense in $\crx {\partial K}$.  But by Lemma~\ref{W}, $H(K)$ is contained in the real-linear span of $\log|\sO(K)^{-1}|$.  Thus the real-linear span of $\log|\sO(K)^{-1}|$, restricted to $\partial K$, is dense in $\crx {\partial K}$.
\epf

\blem\label{trick}
Let $A$ be a uniform algebra on a compact space $X$.  Suppose that there is a dense subset $\sf$ of $A$ such that every function in $\sf$ has a square root in $A$.  Suppose also that the real-linear span of $\logai$ is dense in $\crx X$.  Then $\logai$ itself is dense in $\crx X$.
\elem

\bpf
It suffices to show that the closure $\ol{\logai}$ of $\logai$ is a real-linear space.

First note that $\logai$ is closed under sums; given $f$ and $g$ in $A^{-1}$, we have $\log|f| + \log|g| = \log|fg|$, and $fg$ is in $A^{-1}$.  
Consequently, $\ol{\logai}$ is also closed under sums.

Next we show that $\ol{\logai}$ is closed under division by $2$.  Given $f\in \sf\cap A^{-1}$, by hypothesis there is $h\in A$ such that $h^2=f$.  Note that $h$ is in $A^{-1}$, and $\log|h| =(1/2)\log|f|$.  Thus division by $2$ takes $\log|\sf\cap A^{-1}|$ into $\logai$.  Since the invertible elements form an open subset in any Banach algebra, we have that $\sf\cap A^{-1}$ is dense in $A^{-1}$, and hence $\log|\sf\cap A^{-1}|$ is dense in $\logai$.  It follows that division by $2$ takes $\ol{\logai}$ into itself.  Consequently the same is true for division by $2^n$ for any positive integer $n$.  Since for every $f$ in $A^{-1}$ and positive integer $m$ we have $\log|f^m|=m\log|f|$ and $\log|f^{-1}|=-\log|f|$, we see that multiplication by any integer takes $\ol{\log|A^{-1}|}$ into itself.  We have now shown that for every dyadic rational $r$ and every function $u$ in $\ol{\logai}$, the function $ru$ is also in $\ol{\logai}$.  It follows that the same is true with the dyadic rational $r$ replaced by an arbitrary real number.  Thus $\ol{\logai}$ is a real-linear space, as desired.
\epf

The next lemma refines \cite[Theorem~1.2]{Izzo1}.  It will be used in the proof of Theorem~\ref{Dirichlet-PX} to obtain that certain sets have piecewise smooth boundaries thereby assuring us of the applicability of Stokes' theorem.
Here, and throughout the paper, $\D$ denotes the open unit disc in the complex plane, and
given a disc $D$, we denote the radius of $D$ by $r(D)$.

\blem\label{specialK}
For each $r>0$, there exists a sequence of open discs $\sequencecurly Dk$ such that the following conditions hold:
\itemskip
\begin{enumerate}
\item[(i)] $\sum_{k=1}^\infty r(D_k)<r$,
\itemskip
\item[(ii)] setting $K=\ol \D\setminus \bigcup_{k=1}^\infty D_k$, the uniform algebra $R(K)$ is nontrivial and strongly regular,
\itemskip
\item[(iii)] each of the sets $\ol \D\sm (D_1\cup \cdots \cup D_n)$, $n\in \Z_+$, has piecewise smooth boundary.
\end{enumerate}
\elem

\bpf
That there exists a sequence $\sequencecurly Dk$ such that conditions~(i) and~(ii) are satisfied is \cite[Theorem~1.2]{Izzo1}.  We explain here how to refine the construction given in \cite{Izzo1} to obtain condition~(iii) as well.

Without loss of generality we assume $r<1$.  The sequence $\sequencecurly Dk$ in \cite[Theorem~1.2]{Izzo1} is obtained by choosing two sequences of open discs $\{D_k^I\}_{k=1}^\infty$ and $\{D_k^W\}_{k=1}^\infty$ such that $\sum r(D_k^I)<r/2$ and $\sum r(D_k^W)<r/2$ and such that setting $K_1=\ol \D\setminus \bigcup D_k^I$ and $K_2=\ol \D\setminus \bigcup D_k^W$ we have $\ol{J_x(R(K_1))} \supset M_x(R(K_1))^2$ 
for all $x\in K_1$ and 
$\ol{M_x(R(K_2))^2}=M_x(R(K_2))$ for all $x\in K_2$, and then taking $\{D_k\}$ to be an enumeration of the collection of discs $\{D_k^I\} \cup \{D_k^W\}$.  It is shown in \cite{Izzo1} that then conditions~(i) and~(ii) hold.  That condition~(iii) can be made to hold as well will be established once we show that the collection $\{D_k^I\} \cup \{D_k^W\}$ can be chosen such that no three circles in the collection $\{\pD\} \cup \{\pD_k^I\} \cup \{\pD_k^W\}$ have a point in common and each intersection between a pair of circles in this collection is a transverse intersection.

The discs $\{D_k^W\}$ were found by Wermer in \cite{W2}.  These discs have disjoint closures and their closures are disjoint from $\pD$, so the corresponding bounding circles are all disjoint.  Thus we need only to show that the discs $\{D_k^I\}$ can be chosen such that their boundaries have only transverse intersections with each other and with the circles in the collection $\{\pD\} \cup \{\pD_k^W\}$.  

The discs $\{D_k^I\}$ were chosen in \cite{Izzo1} as follows.  For each $n\in \Z_+$, we let $\vep_n>0$, $\eta_n>0$, and $\sigma_n=\sigma(\vep_n, \eta_n)>0$ be as in the proof of \cite[Theorem~3.5]{Izzo1}.  The square $[-1,1]\times[-1,1]\supset \ol \D$ can be expressed as the union of $M_n\leq\bigl( \lceil 2/\sigma_n \rceil \bigr)^2 \leq 9/\sigma_n^2$ squares $S_n^1,\ldots, S_n^{M_n}$ of side length $\sigma_n$.  Each such square $S_n^j$ is contained in the open disc $\Delta(a, \sigma_n)$ with radius $\sigma_n$ and center $a=a_n^{(j)}$, where $a_n^{(j)}$ is any point within a distance less than $(1-\sqrt{2}/2)\sigma_n$ of the center of the square $S_n^j$.  To the disc $\Delta(a, \sigma_n)$ is associated, by \cite[Corollary~3.3]{Izzo1}, a sequence of open discs $\{D_k^a\}$.  The collection $\{D_k^I\}$ is the union of all the collections $\{D_k^a\}$ as $n$ ranges over $\Z_+$ and we consider each square $S_n^1,\ldots, S_n^{M_n}$.  In \cite{Izzo1} each point $a=a_n^{(j)}$ was taken (somewhat implicitly) to be the center of the corresponding square $S_n^j$.  Our desired transversality can be achieved by (possibly) making other choices.  If the center $a$ of the disc $\Delta(a,\sigma_n)$ is translated, the associated open discs $\{D_k^a\}$ are translated together along with $a$.  Enumerate all the squares $\{S_n^j\}$.  We will choose the points $a=a_n^{(j)}$ inductively.

Given two circles in the plane, one of which is fixed and the other of which is to be translated, the set of vectors by which the second circle can be translated so as to have non-transverse intersection with the first is the union of two concentric circles (or one circle and a point in the case when the fixed circle and circle to be translated have the same radius) and hence is a set of measure zero in the plane.  Since there are only a countable number of circles involved, it follows that at each stage in the induction, we can choose the new center $a_n^{(j)}$ in such a way that each of the new boundary circles associated with the discs $\{D_k^{a_n^{(j)}}\}$ has only transverse intersections with $\pD$ and with each of the boundary circles associated with $\{D_k^W\}$ and with boundary circles associated with the discs of the form $D_k^I$ chosen earlier in the induction.  Thus the desired transversality can be achieved.
\epf

\blem\label{uniqueness}
Let $K$ be a compact planar set such that $R(K)$ is normal.  Then the only Arens-Singer measures for $R(K)$ are the point masses.
\elem

\bpf
Since $R(K)$ is normal, $K$ has empty interior, and hence $R(K)$ has dense invertibles.  It is well known that on a normal uniform algebra the only \emph{Jensen} measures are the point masses (see the proof of \cite[Theorem~27.3]{S1}).  It is easily shown, by an application of Fatou's lemma, that on a uniform algebra with dense invertibles, every Arens-Singer measure is a Jensen measure
\epf

The following lemma is a variation on \cite[Lemma~2.1]{CGI}, which is in turn a generalization of the theorem that for $X$ a compact set in $\C^n$, the uniform algebra $R(X)$ is generated by $n+1$ functions.

\begin{lemma}\label{generators}
Let $A$ be a uniform algebra on a compact space $X$.  Suppose that $A$ is generated by a sequence of functions ${\mcal F}=\{f_1, f_2, \ldots\}$.  Suppose also that there is a positive integer $N$ such that for each $n>N$, the function $f_n$ either is the inverse of a function in the uniform algebra $[\row f{{n-1}}]$ generated by $\row f{n-1}$ or is a logarithm of a function in the uniform algebra 
$[\row f{{n-1}}]$.  Then $A$ is generated by $N+1$ functions.
\end{lemma}

In the proof we will use of the following very well-known, elementary fact \cite[Theorem~1.2.1]{Browder}.

\begin{proposition}\label{elem}
Let $B$ be a Banach algebra with identity element $e$.  If $f\in B$, $\lambda\in \C$, and $\|f\|<|\lambda|$, then $\lambda e-f$ is invertible in $B$.
\end{proposition}

\begin{proof}[Proof of Lemma~\ref{generators}]
The proof is similar to the proof of \cite[Lemma~2.1]{CGI} but somewhat easier.

We will show that if $(c_n)$ is a sequence of positive numbers decreasing to zero sufficiently rapidly, and if $b_{N+1}=\sum_{n=N+1}^\infty c_nf_n$, then $A$ is generated by the functions $f_1,\ldots, f_N, b_{N+1}$.  For $n>N$, in the case when $f_n$ is the inverse of a function in 
 $[\row f{{n-1}}]$, set $h_n = f^{-1}_n$.
Otherwise $f_n$ is a logarithm of a function $g_n$ in $[\row f{{n-1}}]$.
Now choose a sequence $\{c_n: n\geq N+1\}$ of constants such that:
\begin{itemize}
\item[(i)] $c_n > 0$ for $n \geq N+1$.
\smallskip
\item[(ii)] $c_n \|f_n\| < 2^{-n}$ for $n \geq N+1$.
\smallskip
\item[(iii)] 
$c_n \|f_n h_k\| < 2^{-n} c_k$ for $n > k \geq N+1$ and $k$ such that $f_k=h_k^{-1}$.
\smallskip
\item[(iv)]
$\displaystyle\biggl\| \sum_{n = k+1}^{\infty} c_n f_n \biggr\| < \pi c_k$ for $k \geq N+1$ and $k$ such that~$\exp f_k=\nobreak g_k$.
\end{itemize}
\smallskip
Now define, for $k \geq N+1$,
$$b_k = \sum_{n = k}^{\infty} c_n f_n.$$
Condition (ii) assures that each $b_k$ is a well-defined function in $A$.
Let $B = [f_1, \ldots, f_N, b_{N+1}]$.
Clearly $B \subseteq A$.
We will prove that 
$A = B$, by using induction to show that $f_n$ is in $B$ for each $n = 1, 2, 3,  \ldots$.
By definition $\row fN$ are in $B$.
Now given $k\geq N+1$, we assume that $\row f{k-1}$ are in $B$, and show that then $f_k$ is in $B$.
First note that $b_k = b_{N+1} - (c_{N+1} f_{N+1} + \ldots + c_{k-1} f_{k-1})$ is in $B$.
Also $h_k$ is in $B$ since $h_k$ is in $[\row f{k-1}]$.\\

\noindent
Case I: $f_k=h_k^{-1}$.\\
Note that $b_k h_k$ is in $B$, since $h_k$ is in $[\row f{k-1}]$.  Furthermore,
\begin{align*}
\|c_k - b_k h_k\|
&= \biggl \|c_k - \sum_{n = k}^{\infty} c_n f_n h_k\biggr \| \\\noalign{\vskip 4pt}
&=\biggl \|\sum_{n = k+1}^{\infty} c_n f_n h_k\biggr\|
\leq \sum_{n = k+1}^{\infty} c_n \| f_n h_k\| \le 2^{-k} c_k < c_k.
\end{align*}
Therefore, by Proposition~\ref{elem}, $b_k h_k = c_k - (c_k -b_k h_k)$ is invertible in $B$.
Hence $h_k$ is invertible in $B$, that is, $f_k = h^{-1}_k$ is in $B$.\\

\noindent
Case II: $\exp f_k = g_k$.\\
Rearranging the equation $b_k = \sum_{n = k}^{\infty} c_n f_n$ yields
$$ - \sum_{n=k+1}^\infty (c_n/c_k) f_n = f_k - (1/c_k)b_k$$
so
\begin{align*}
\exp\biggl(-\!\sum_{n=k+1}^\infty (c_n/c_k) f_n\biggl) 
&= \exp f_k  \exp\bigl((-1/c_k)b_k\bigr)\\
&=g_k \exp\bigl((-1/c_k)b_k\bigr).
\end{align*}
Therefore, $\exp\Bigl(-\sum_{n=k+1}^\infty (c_n/c_k) f_n\Bigl)$ is in $B$.  By condition (iv),
$$\biggl\| \sum_{n=k+1}^\infty (c_n/c_k) f_n \biggr\| < \pi$$
so the range of $\exp\Bigl(-\sum_{n=k+1}^\infty (c_n/c_k) f_n\Bigl)$ is contained in $\C\sm (-\infty, 0]$.  Therefore, the principle branch of the logarithm function is the uniform limit of polynomials on the range of $\exp\Bigl(-\sum_{n=k+1}^\infty (c_n/c_k) f_n\Bigl)$, and consequently the function $-\sum_{n=k+1}^\infty (c_n/c_k) f_n$ is in $B$.
Since 
$$f_k=(1/c_k)b_k - \sum_{n=k+1}^\infty (c_n/c_k) f_n$$
we conclude that $f_k$ is in $B$.
Thus, by induction, $f_n$ is in $B$ for each positive integer $n$, and hence, $A = B = [\row f{{n-1}}, b_{N+1}]$.
\end{proof}

The straightforward proof of the next lemma is left to the reader.

\blem\label{straightforward}
Let $A$ be a uniform algebra on a compact space $X$.  Let $K$ be a compact subset of $X$, and let $x$ be a point of $K$.  If $A$ is strongly regular at $x$, then $\ol {A|K}$ is strongly regular at $x$ as well.
\elem

%
%
%
%

\section{Proof of Theorem~\ref{logmodular2}}\label{logmodular2-section}

\noindent{\emph{Step 1: We describe a setup that will yield the desired uniform algebra.}}

Give the set $\{(m,n)\in \Z_+^2:1\leq n\leq m\} \cup \{(0,0)\}$ the dictionary order, and let 
$\sigma:\{(m,n)\in \Z_+^2:1\leq n\leq m\} \cup \{(0,0)\} \ra \Z_+\cup \{0\}$ be the order-preserving bijection between the given sets.  

We will show that the following setup can be achieved.
\itemskip
\begin{enumerate}
\item[(1)] There is sequence $\sequence fn$ with each $f_n$ a polynomial on $\C^n$.
\itemskip
\item[(2)] There is a sequence $\sequence \alpha n$ in $\C$ with $|\alpha_n|< 1/n$ for all $n\in\Z_+$.
\itemskip
\item[(3)] Set $\Sigma_0=\C^1$.  For each $N\in\Z_+$, set
\begin{align*}
\Sigma_N=\{\vector z{N+1} \in \C^{N+1}&: z_{n+1}^2=f_n\vector zn\cr
&{\rm for\ every\ } n=1,\ldots, N\}.
\end{align*}
Each $\Sigma_N$
is a possibly disconnected open Riemann surface contained in $\C^{N+1}$.
\itemskip
\item[(4)] Set $X_0=\ol \D$.  For each $N\in\Z_+$, set 
$$X_N=\{\vector z{N+1}\in\Sigma_N:z_1\in \ol \D\}.$$
\itemskip 
\item[(5)] For each $N\in\Z_+\cup\{0\}$, set
$$A_N=P(X_N).$$
There is a countable set of polynomials $\doubleindex gNj$ that is dense in $A_N$.
\itemskip
\item[(6)] Given $N\in\Z_+$, let $(r,s)=\sigma^{-1}(N-1)$.\hfil\break
\leftline{If $r=s$, then $f_N\vector zN = g_{0, r+1}(z_1)-\alpha_N$.}
\leftline{If $r>s$, then $f_N\vector zN = g_{\sigma(s,s), r-s}\vector z{\sigma(s,s) + 1}-\alpha_N$.}
\end{enumerate}

Note that each $\Sigma_N$ projects onto $\Sigma_{N-1}$ under the map $\vector z{N-1}$, and hence onto $\Sigma_0=\C$ under $z_1$.

Set
$$\Sigmao = \{ \sequence zn \in\Co : z_{n+1}^2 = f_n\vector zn 
{\rm\ for\ every\ }n\in\Z_+\}$$
and
$$\Xo = \{\sequence zn \in \Sigmao : z_1\in \ol \D \}.$$
Let $\Ao = P(\Xo)$.
After showing that the above setup can be achieved, we will show that $\Ao$ satisfies the requirements of the theorem.

\medskip
\noindent{\emph{Step 2: We prove that the setup above can be achieved.}

We apply induction.  
The objects $\Sigma_0$, $X_0$, and $A_0$ are explicitly specified in conditions~(3)--(5).
Choose a countable set of polynomials 
$\doubleindex g0j$
that is dense in $A_0$.
Then noting that there is no $f_0$ and no $\alpha_0$ to be specified, we see that everything is as required for $N=0$.

Now let $M\in\Z_+$ be arbitrary and assume that the objects $f_n$ and $\alpha_n$, for all $1\leq n<M$, and $\Sigma_n$, $X_n$, $A_n$, and $\{g_{n,j}\}_{j=1}^\infty$, for all $0\leq n<M$, have been obtained consistent with the requirements in conditions (1)--(6).  Let $(r,s)=\sigma^{-1}(M-1)$.  If $r=s$, define the polynomial $q_M$ on $\C^M$ by $q_M\vector zM = g_{0, r+1}(z_1)$.  If $r>s$, define the polynomial $q_M$ on $\C^M$ by $q_M\vector zM=g_{\sigma(s,s), r-s}\vector z{\sigma(s,s)+1}$.  By Sard's theorem, almost every point of $\C$ is a regular value of the map $z_{M+1}^2 - q_M\circ\vector zM : \Sigma_{M-1} \times \C \ra \C$.  Thus for some $\alpha_M\in\C$ with $|\alpha_M|<1/M$, setting $f_M=q_M - \alpha_M$ yields that $z_{M+1}^2 - f_M \circ \vector zM : \Sigma_{M-1} \times \C \ra \C$ has $0$ as a regular value.  Since the set $\Sigma_M$ defined as in (3) satisfies 
$$\Sigma_M = \{\vector z{M+1}\in \Sigma_{M-1} \times \C : z_{M+1}^2 = f_M\vector zM \},$$
the set $\Sigma_M$ is then a possibly disconnected open Riemann surface.
Thus the induction can proceed (since we can choose a countable set of polynomials 
$\doubleindex gMj$
that is dense in $A_M=P(X_M)$).

\medskip
\noindent{\emph{Step 3: We prove that $\Ao$ has a dense set of elements 
having square roots.}}

For each $1\leq n\leq N\leq\omega$, we will identify each function on $\C^n$ with the corresponding function on $\C^N$ that is independent of the extra variables.  Then each $A_N$ is a subset of $\Ao$, and each polynomial in $\Ao$ is a polynomial in $A_N$ for some finite $N$.  Thus the increasing union $\bigcup_{N=0}^\infty A_N$ is dense in $\Ao$, and hence so is the increasing union $\bigcup_{s=0}^\infty A_{\sigma(s,s)}$.  Since $\{f_n+\alpha_n\}_{n=1}^\infty = \bigcup_{s=0}^\infty \{g_{\sigma(s,s),j} \}_{j=1}^\infty$, and $\alpha_n\ra 0$ as $n\ra \infty$, it follows that $\{f_n\}_{n=1}^\infty$ is dense in $\Ao$.  Since $f_n=z_{n+1}^2$ (when regarded as elements of $\Ao$), this shows that $\Ao$ has a dense set of elements 
having square roots.

\medskip
\noindent{\emph{Step 4: We prove that $\Ao$ has proper Shilov boundary.}}

For each finite $N$, set 
$$\Gamma_N = \{ \vector z{N+1} \in \Sigma_N : z_1\in \pD \},$$
and set 
$$\Gamma_\omega = \{\sequence zn \in \Sigma_\omega : z_1\in \pD\}.$$
Note that for each $N$ finite, $X_N$ is a compact subset of $\Sigma_N$ with boundary $\Gamma_N$.  Consequently, every polynomial in $A_N$ takes its maximum modulus on $\Gamma_N$.  Therefore, every polynomial in $\Ao$ takes its maximum modulus on $\Gamma_\omega$.  It follows that $\Gamma_\omega$ is a boundary for $\Ao$, so $\Ao$ has proper Shilov boundary.

\medskip
\noindent{\emph{Step 5: We prove that for  each $N\in \Z_+\cup\{0\}$, each point of $\Gamma_N$ has a unique representing measure with respect to $A_N$.}}

For $N=0$, this is well known (and easily verified).  Assume, for the purpose of induction, that the assertion holds for a particular value of $N$.  Let $x\in \Gamma_{N+1}$ be arbitrary, and let $\mu$ be a representing measure for $x$ with respect to $A_{N+1}$.  Let $\pi: X_{N+1}\ra X_N$ be projection on the first $N$ coordinates.  Then the push forward measure $\pi_*(\mu)$ (defined by $\pi_*(\mu)(E) = \mu(\pi^{-1}(E))$ for every Borel set $E$ in $X_N$) represents $\pi(x)$ with respect to $A_N$.  Since $\pi(x)$ is in $\Gamma_N$, our induction hypothesis implies that $\pi_*(\mu)$ is the unit point mass at $\pi(x)$.  Consequently, $\mu$ is supported on the fiber over $\pi(x)$.  But each fiber of $\pi$ consists of at most two points.  Consequently, $\mu$ must be supported on the single point $x$, so $\mu$ is the unit point mass at $x$.

\medskip
\noindent{\emph{Step 6: We prove that for each $N\in \Z_+\cup\{0\}$, the equality $P(X_N)=\sO(X_N)$ holds.}}

Let $f$ be an arbitrary function holomorphic on a neighborhood $V$ of $X_N$ in $\Sigma_N$.  There exists in $\C$ an open disc $\Omega$  that contains $\ol \D$ such that
$$X_N\subset (\Omega\times \C^N) \cap \Sigma_N \subset V.$$
By \cite[Theorem~I5(ii)]{Gunning}, $f$ extends to a holomorphic function $\tilde f$ on $\Omega\times \C^N$.  Since $\tilde f$ can be approximated by polynomials uniformly on compact subsets of $\Omega\times \C^N$, the desired equality follows.

\medskip
\noindent{\emph{Step 7: We prove that $\Ao$ is logmodular on $\Gamma_\omega$ thereby completing the proof.}}

Together, Steps~5 and~6 and Lemma~\ref{density} yield that, for each finite $N$, the real-linear span of $\log|A_N^{-1}|$, restricted to $\Gamma_N$, is dense in $\crx{\Gamma_N}$.  Regarding the functions in $\crx{\Gamma_N}$ as functions on $\Gamma_\omega$ independent of the extra variables, the union $\bigcup_{N=0}^\infty \crx{\Gamma_N}$ is dense in $C(\Gamma_\omega)$ (by the Stone-Weierstrass theorem for instance).  It follows that the real-linear span of $\log|A_\omega^{-1}|$, restricted to $\Gamma_\omega$, is dense in $\crx{\Gamma_\omega}$.  Therefore, $\log|A_\omega^{-1}|$, restricted to $\Gamma_\omega$, itself is dense in $\crx{\Gamma_\omega}$ by Lemma~\ref{trick}, and the proof is complete.

%
%
%
%

\section{Proof of Theorem~\ref{Dirichlet-PX}}\label{Dirichlet-PX-section}

\noindent{\emph{Step 1: We describe a setup that will yield the desired uniform algebra.}}

As in the proof of Theorem~\ref{logmodular2}, we
give the set $\{(m,n)\in \Z_+^2:1\leq n\leq m\} \cup \{(0,0)\}$ the dictionary order, and we let 
$\sigma$ be the order-preserving bijection from this set onto  $\Z_+\cup \{0\}$.  Let $\sequencecurly Dk$ and $K$ be as in Lemma~\ref{specialK} with $0<r<1$.  Let $\area$ denote the $2$-dimensional Lebesgue measure of $K$, and note that $\area>0$.   

We will show that the following setup can be achieved.
\itemskip
\begin{enumerate}
\item[(1)] There is sequence $\sequence fn$ with each $f_n$ a rational function with domain $U_n \subset \C^n$.
\itemskip
\item[(2)] There is a sequence of real numbers $\sequence cn$ and a sequence of positive integers $\sequence mn$.
\itemskip
\item[(3)] Set $\Sigma_0=\C^1$.  For each $N\in\Z_+$, set
\begin{align*}
\phantom{MMm}\Sigma_N=\{&\vector z{N+1} \in \C^{N+1}: \vector zn\in U_n {\rm\ and}\cr
&\exp z_{n+1}=f_n\vector zn
{\rm\ for\ every\ } n=1,\ldots, N\}\end{align*}
Each $\Sigma_N$ is a possibly disconnected open Riemann surface contained in $\C^{N+1}$.
\itemskip
\item[(4)] Set $X_0=K$.  For each $N\in\Z_+$, set 
\begin{align*}
\phantom{MMM}X_N=\{\vector z{N+1}\in\Sigma_N &:
\vector zN\in X_{N-1}\ {\rm\ and}\cr
&\phantom{M}\Im z_{N+1}\in [c_N, c_N + 2\pi m_N]\}.
\end{align*}
The inclusion $U_N\supset X_{N-1}$ holds, and $f_N$ is zero-free on $X_{N-1}$.
\itemskip
\item[(5)]
For each $N\in\Z_+\cup \{0\}$, set 
$$A_N=R(X_N).$$
There is a countable set of rational functions $\doubleindex gNj$ that is dense in $A_N$ and such that each function $g_{N,j}$ is zero-free on $X_N$.
\itemskip
\item[(6)] Given $N\in\Z_+$, let $(r,s)=\sigma^{-1}(N-1)$.\hfil\break
\leftline{If $r=s$, then $f_N\vector zN = g_{0, r+1}(z_1)$.}
\leftline{If $r>s$, then $f_N\vector zN = g_{\sigma(s,s), r-s}\vector z{\sigma(s,s) + 1}$.}
\itemskip
\item[(7)] For each $N\in \Z_+ \cup \{0\}$, there are a positive integer $k(N)$
and a decreasing sequence of compact sets $\{X_N^k\}_{k=k(N)}^\infty$ in $\Sigma_N$ such that $\bigcap_{k=k(N)}^\infty X_N^k = X_N$ and each $X_N^k$ has piecewise smooth boundary in $\Sigma_N$.
\itemskip
\item[(8)] For each $N\in \Z_+\cup\{0\}$ and $k\geq k(N)$, let $\mu_N^k$ be the complex regular Borel measure on $\partial X_N^k$ induced by the form $dz_1$, i.e., such that $\int f\, d\mu_N^k = \int_{\partial X_N^k} f\, dz_1$ for every $f\in C(\partial X_N^k)$.
Then there is a $\delta_N>0$ independent of $k$ such that
$$\|\mu_N^k\|<4\pi\, m_1\cdots m_N - \delta_N.$$
\itemskip
\item[(9)] For each $N\in \Z_+\cup\{0\}$ and $k\geq k(N)$, 
$$\left| \int \ol z_1 \, d\mu_N^k\right| > 2\area\,  m_1\cdots m_N.$$
\end{enumerate}
(When $N=0$ the right hand sides of the inequalities in conditions~(8) and~(9) are to be interpreted as $4\pi - \delta_0$ and $ 2\area$, respectively.)

Note that for every $0\leq N\leq M\leq \omega$, the canonical projection $\C^M\ra \C^N$ maps $X_M$ onto $X_N$.

Set
\begin{align*}
\Sigmao = \{&\sequence zn \in\Co :\vector zN\in U_N {\rm\ and}\cr
&\exp z_{N+1}=f_N\vector zN
{\rm\ for\ every\ } N\in\Z_+\},
\end{align*}
set
$$\Xo = \{\sequence zn \in \Sigmao : \vector zN\in X_N {\rm\ for\ every\ } N\in\Z_+\},$$
and set
$$\Ao=R(\Xo).$$

After showing that the above setup can be achieved, we will show that then 
$\Ao$ satisfies all the conditions required to yield the theorem.  The proof will show that conditions~(1) through (6) imply that $\Ao$ satisfies all the required conditions except for the nontriviality of $\Ao$.  Conditions (7) through (9) will be used to show that $\Ao$ is nontrivial.

\medskip
\noindent{\emph{Step 2: We prove that the setup above can be achieved.}

We apply induction.
The objects $\Sigma_0$, $X_0$, and $A_0$ are explicitly specified in conditions~(3)--(5).
Because $X_0$ has empty interior in $\C$, there exists a countable set of rational functions $\doubleindex g0j$
that is dense in $A_0$ and such that each function $g_{0,j}$ is zero-free on $X_0$.
For each $k\geq 1$, set $X_0^k= \ol \D\sm \bigcup_{j=1}^k D_j$, and let $\mu_0^k$ be the complex regular Borel measure on $\partial X_0^k$ induced by the form $dz_1$.  The length of $\partial X_0^k$ is less than $2(1+r)\pi$, so $\|\mu_0^k\|< 2(1+r)\pi < 4\pi - \delta_0 < 4\pi$ for some $\delta_0>0$.  An application of Green's theorem (Stokes' theorem) carried out in \cite[p.~270]{S1} shows that 
$$\left | \int \ol z_1 \, \mu_0^k \right | =  2\, {\rm Area}(X_0^k)  > 2\area.$$
Thus noting that there are no $f_0$, $U_0$, $c_0$, and $m_0$ to be specified, we see that everything is as required for $N=0$.

Now let $M\in \Z_+$ be arbitrary, and assume that the objects $f_n$, $U_n$, $c_n$, and $m_n$, for all $1\leq n<M$, and 
$\Sigma_n$, $X_n$, $A_n$, $\{g_{n,j}\}_{j=1}^\infty$, $\{X_n^k\}_{k=k(n)}^\infty$, $\mu_n^k$, and $\delta_n$, for all $0\leq n<M$, have been obtained consistent with the requirements in conditions~(1)--(9).  
Define the rational function $f_M$ in accordance with condition~(6).
Note that the domain $U_M$ of $f_M$ contains $X_{M-1}$, and that $f_M$ is zero-free on $X_{M-1}$.  
Note also that zero (and in fact every complex number) is a regular value of the function 
 $\exp z_{M+1} - f_M \circ \vector zM : (\Sigma_{M-1} \cap U_M) \times \C \ra \C$.
The set $\Sigma_M$ defined in accordance with condition~(3) satisfies
$$\Sigma_M=\{\vector z{M+1} \in (\Sigma_{M-1}\cap U_M) \times \C:
\exp z_{M+1}=f_M\vector zM\}$$
and hence is a possibly disconnected open Riemann surface.

We now turn to obtaining suitable values for $c_M$ and $m_M$.  Choose $k(M)$ large enough that $X_{M-1}^k \subset U_M$ for all $k\geq k(M)$.  Then for $k\geq k(M)$, set
$$\tX_M^k = \{\vector z{M+1} \in \Sigma_M : \vector zM \in X_{M-1}^k\}.$$

Since $\Sigma_M$ is locally a graph over $\Sigma_{M-1}\cap U_M$, and $X_{M-1}^k\subset U_M$ has piecewise smooth boundary in $\Sigma_{M-1}$, the set $\tX_M^k$ has piecewise smooth boundary in $\Sigma_M$.  Applying the transversality theorem (see for instance \cite[pp.~68--69]{GP}) yields that for almost every real number $c$, the manifold $\{\Im z_{M+1} =c\}$ in $\C^{M+1}$ is transverse to $\Sigma_M$ and, for each $k\geq k(M)$,  avoids each of the points where $\partial X_{M-1}^k$ is not smooth and is transverse to each of the smooth $1$-manifolds that make up $\partial X_{M-1}^k$.  Choose $c_M$ to be such a real number $c$.  Then for every positive integer $m$, the set 
\begin{align*}
X_M^k(m)=\{\vector z{M+1}&\in\Sigma_M :  \vector zM\in X_{M-1}^k\ {\rm\ and}\cr
&\Im z_{M+1}\in [c_M, c_M + 2\pi m]\}
\end{align*}
has piecewise smooth boundary in $\Sigma_M$.

Note that $\partial X_M^k(m)= \eik \cup \eck$, where
\begin{align*}
\eik=\{\vector z{M+1}&\in\Sigma_M :  \vector zM\in\partial X_{M-1}^k\ {\rm\ and}\cr
&\Im z_{M+1}\in [c_M, c_M + 2\pi m]\}
\end{align*}
and 
\begin{align*}
\eck=\{\vector z{M+1}&\in\Sigma_M :  \vector zM\in X_{M-1}^k\ {\rm\ and}\cr
&\Im z_{M+1}\in \{c_M, c_M + 2\pi m\}\}.
\end{align*}
For each $s\in\Z_+$, the canonical projection $\C^{M+1}\ra \C^M$ takes 
$$
\bigl\{\vector z{M+1}\in\Sigma_M : 
\Im z_{M+1}\in \bigl(c_M + 2\pi s, c_M + 2\pi (s+1)\bigr)\bigr\}
$$
bijectively onto
$$\{\vector zM\in\Sigma_{M-1}\cap U_M : 
\arg f_M\vector zM \neq c_M\}.
$$
Consequently, setting 
$$\eik(s) = \eik \cap \bigl\{\Im z_{M+1} \in \bigl(c_M + 2\pi s, c_M + 2\pi(s+1)\bigr)\bigr\},$$ 
our choice of $c_M$ insures that for each $s=0,\ldots m-1$ and for each continuous function $f$ on $\Sigma_{M-1}$,
$$\int_{\eik(s)} f\, dz_1 =
\int_{\partial X_{M-1}^k}\! f \, dz_1,$$
and hence,
$$\int_{\eik} f \, dz_1 = m\, \int_{\partial X_{M-1}^k} f \, dz_1.$$
In particular,
$$(**)\qquad\ \  \left | \int_\eik \ol z_1 \, dz_1 \right | = m\left | \int_{\partial X_{M-1}^k} \ol z_1 \, dz_1 \right | > 2\area\,  m_1\cdots m_{M-1} m,\phantom{(****)}$$
and the norm of the measure $\mu_I^k$ on $\eik$ induced by $dz_1$ satisfies 
$$\| \mu_I^k \| = m \| \mu_{M-1}^k \| < 4\pi\, m_1\cdots m_{M-1} m - m\delta_{M-1}.$$

By our choice of $c_M$, the subset of $\Sigma_M$ where $\Im z_{M+1}$ is either $c_M$ or $c_M+2\pi m$ is a smooth $1$-manifold, and $\eck$ is a finite union of closed arcs on that $1$-manifold.  Thus the norm of the measure $\mu_J^k$ on $\eck$ induced by $dz_1$ is finite.  Furthermore, $\eck$ decreases with $k$, and hence so does $\|\mu_J^k\|$.  Thus there is a constant $C$ such that $\| \mu_J^k\|\leq C$ for all $k\geq k(M)$.  Note that $C$ is independent of $m$.  
Now the measure $\mu_M^k(m)$ on $\partial X_M^k(m)$ induced by $dz_1$ satisfies
$$\|\mu_M^k(m)\| \leq \|\mu_I^k\| + \| \mu_J^k\| \leq 4\pi\, m_1\cdots m_{M-1} m - m\delta_{M-1} +C.$$
Therefore, choosing the positive integer $m_M$ sufficiently large and setting $X_M^k=X_M^k(m_M)$ yields
that there is a $\delta_M>0$ independent of $k$ such that
$$\|\mu_M^k\| < 4\pi\, m_1\cdots m_M - \delta_M.$$

The set $\eck$ is the union of two disjoint sets $E(1)$ and $E(2)$ where $\Im z_{M+1}$ takes the values $c_M$ and $c_M+2\pi m_M$, respectively.  The canonical projection $\C^{M+1}\ra \C^M$ takes each of the sets $E(1)$ and $E(2)$ 
bijectively onto
$$\{\vector zM\in   X_{M-1}^k : 
\arg f_M\vector zM = c_M\},
$$
but with opposite orientations (when regarded as subsets of $\partial X_M^k$).  Consequently,  
$$\int_\eck \ol z_1 \, dz_1 = 0,$$
so $(**)$ gives
$$\left | \int_{\partial X_M^k} \ol z_1 \, dz_1  \right | = \left | \int_\eik \ol z_1 \, dz_1 \right |
> 2\area\,  m_1\cdots m_{M-1} m_M.$$

With $X_M$ defined in accordance with condition~(4), the equation $\bigcap_{k=k(M)}^\infty X_M^k = X_M$ holds.

All that remains to complete the inductive step is to verify that a set $\{g_{M,j}\}_{j=1}^\infty$
as in condition~(5) can be chosen.  For that it is enough to show that the uniform algebra $A_M=R(X_M)$ has dense invertibles.  Since the image of $X_M$ under $z_1$ has empty interior in $\C$, any connected open set of $\Sigma_M$ contained in $X_M$ must lie in a single fiber of $z_1$ restricted to $\Sigma_M$.  Since each such fiber is at most countable, this implies that $X_M$ has empty interior in $\Sigma_M$.  Consequently, the image of $X_M$ under any smooth function on $\Sigma_M$, and in particular any rational function holomorphic on $X_M$, has empty interior in $\C$ (by \cite[Lemma~4.3]{Izzo2018}, for instance).  That $A_M$ has dense invertibles follows.

\medskip
\noindent{\emph{Step 3: We prove that $\Ao$ is nontrivial.}

First note that (by Stokes' theorem) for any finite $N$ and $k\geq k(N)$, the measure $\mu_N^k$ annihilates every function in $\sO(X_N^k)$ and hence annihilates every function that is holomorphic in a neighborhood of $X_N^k$ in $\C^N$.  Now consider an arbitrary rational function $g$ holomorphic on a neighborhood of $\Xo$ in $\Co$.  For some finite $N$, the function $g$ can be regarded as a rational function holomorphic on a neighborhood of $X_N$ in $\C^N$.  Then (by condition~(7)) for all $k$ sufficiently large, $g$ is holomorphic on a neighborhood of $X_N^k$ in $\C^N$.  By conditions~(8) and~(9), for these $k$,
$$\|\ol z_1 - g \|_{X_N^k} \geq \frac{1}{\|\mu_N^k \|}\, {\left | \int \ol z_1 \, d\mu_N^k \right |} > \frac{\area}{2\pi}.$$
Since $\bigcap_{k=k(N)}^\infty X_N^k = X_N$, this implies
$$\|\ol z_1 - g \|_{X_N}  \geq \frac{\area}{2\pi}.$$
Therefore, $\| \ol z_1 - f \|_{\Xo} \geq \area/2\pi>0$ for every rational function $f$ holomorphic on a neighborhood of $\Xo$ in $\Co$.  Thus $\ol z_1$ is not in $\Ao$, so $\Ao$ is nontrivial.

\medskip
\noindent{\emph{Step 4: We prove that $\Ao$ has dense exponentials.}}

For each $0\leq n\leq N\leq\omega$,  we can regard $A_n$ as a subset of $A_N$ by identifying each function in $A_n$ with the corresponding function in $A_N$ that is independent of the extra variables.  Then the increasing union $\bigcup_{N=0}^\infty A_N$ is dense in $\Ao$, and hence so is the increasing union $\bigcup_{s=0}^\infty A_{\sigma(s,s)}$.  Since $\{f_n\}_{n=1}^\infty = \bigcup_{s=0}^\infty \{g_{\sigma(s,s),j} \}_{j=1}^\infty$, it follows that $\{f_n\}_{n=1}^\infty$ is dense in $\Ao$.  Since $f_n=\exp z_{n+1}$ (when regarded as elements of $\Ao$), this shows that $\Ao$ has dense exponentials.

\medskip
\noindent{\emph{Step 5: We prove that for each $N\in \Z_+\cup\{0\}$, the equality $R(X_N)=\sO(X_N)$ holds.}}

The inclusion $R(X_N)\subset \sO(X_N)$ is obvious.  To prove the reverse inclusion
let $f$ be an arbitrary function holomorphic on a neighborhood $V$ of $X_N$ in $\Sigma_N$.  The set $X_N$ is rationally convex.  Therefore, there is a domain of holomorphy $\Omega$ in $\C^{N+1}$ such that
$$X_N\subset \Omega \cap \Sigma_N \subset V.$$
By \cite[Theorem~I5(ii)]{Gunning}, $f$ extends to a holomorphic function $\tilde f$ on $\Omega$.
By the rationally convex analogue of the Oka-Weil theorem \cite[Theorem~29.9]{S1}, 
$\tilde f$ can be approximated uniformly on $X_N$ by rational functions.

\medskip
\noindent{\emph{Step 6: We prove that, for each finite $N$, the only Arens-Singer measures for $A_N$ are the unit point masses}}

For $N=0$, the assertion is immediate from Lemma~\ref{uniqueness}.
Assume, for the purpose of induction, that the assertion holds for a particular value of $N$.  Let $x\in X_{N+1}$ be arbitrary, and let $\mu$ be an Arens-Singer measure for $x$ with respect to $A_{N+1}$.  Let $\pi: X_{N+1}\ra X_N$ be projection on the first $N$ coordinates.  Then the push forward measure $\pi_*(\mu)$ is an Arens-Singer measure for $\pi(x)$ with respect to $A_N$.  Therefore, our induction hypothesis implies that $\pi_*(x)$ is the unit point mass at $\pi(x)$.  Consequently, $\mu$ is supported on the fiber over $\pi(x)$.  But each fiber of $\pi$ is finite.  Consequently, $\mu$ must be supported on the single point $x$, so $\mu$ is the unit point mass at $x$.

\medskip
\noindent{\emph{Step 7: We prove that $\Ao$ is Dirichlet on $\Xo$.}}

Note that $\exp \Ao= \Ao^{-1}$, because $\Ao$ has dense exponentials and $\exp A$ is closed in $A^{-1}$ for any commutative Banach algebra $A$.  Consequently, the equality 
$\log|\Ao^{-1}| = \Re(\Ao)$ holds, since for any $f\in \Ao$ we have $\log|\exp(f)|=\Re(f)$.  Therefore, it suffices to show that $\log|\Ao^{-1}|$ is dense in $\crx {\Xo}$.

Recall from the last paragraph of Step~2 that each $X_N$ has empty interior in $\Sigma_N$.  Therefore, by the results of the two preceding steps
and Lemma~\ref{density}, the real-linear span of $\log| A_N^{-1}|$ is dense in $\crx {X_N}$ for each finite $N$.  Regarding all functions as defined on $\Xo$, the union $\bigcup_{N=1}^\infty \crx{X_N}$ is dense in $\crx{\Xo}$.  Consequently, the real-linear span of $\log|\Ao^{-1}|$ is dense in $\crx{\Xo}$.  But $\log|\Ao^{-1}|$ is itself a real-linear space, since $\log|\Ao^{-1}| = \Re(\Ao)$.

\medskip
\noindent{\emph{Step 8: We prove that the maximal ideal space $\mao$ of $\Ao$ is $\Xo$.}

Let $\varphi\in\mao$ be arbitrary.  For each finite $N$, the set $X_N$ is rationally convex, so there exists a unique point $\bigl(\zeta_1^{(N)},\ldots, \zeta_{N+1}^{(N)}\bigr)$ in $X_N$ such that $\varphi(f)= f\bigl(\zeta_1^{(N)},\ldots, \zeta_{N+1}^{(N)}\bigr)$ for every $f\in A_N$.  Given $M>N$, necessarily $(\zeta_1^{(M)},\ldots, \zeta_{N+1}^{(M)}) = (\zeta_1^{(N)},\ldots, \zeta_{N+1}^{(N)})$.  Consequently, setting $\zeta_n=\zeta_n^{(n-1)}$ for every $n\in \Z_+$, the point $\sequence \zeta n$ is in $\Xo$ and satisfies $\varphi(f) = f\bigl(\sequence \zeta n\bigr)$ for every $f\in \Ao$.

\medskip
\noindent{\emph{Step: 9: We prove that $\dim \Xo = 1$.}}

Since $X_0=K$ has empty interior in $\C$, we have that $\dim X_0\leq 1$ (by~\cite[Theorem~IV~3]{HW}).  Let $\pi:\Xo\ra X_0$ be projection onto the first coordinate.  It is easily seen that the inverse image of each point of $X_0$ under $\pi$ is a totally disconnected compact space and hence of topological dimension zero.  Therefore, $\dim \Xo\leq \dim X_0 \leq1$, by \cite[Theorem~VI~7]{HW}.  Since $\dim \Xo$ can not be $0$ (because $\Ao$ is nontrivial), $\dim \Xo$ must be $1$.

\medskip
\noindent{\emph{Step 10: We prove that $A_N$ is regular on $X_N$ for every $1\leq N\leq \omega$.}}

The proof is similar to Feinstein's proof that the Cole construction preserves regularity \cite[Theorems~2.4 and~2.8]{F1}.  

We first apply induction to show that $A_N$ is regular on $X_N$ for every finite $N$.  By hypothesis, $A_0$ is regular.  Now assuming that $A_{N-1}$ is regular, we show that $A_N$ is regular.  Fix $x=\vector x{N+1}$ in $X_N$.  It suffices to show that given $y=\vector y {N+1}$ in $X_N\sm \{\vector x{N+1}\}$, there exists a function $h$ in $J_y(A_N)$ with $h(x)=1$.  

First suppose that $\vector y N \neq \vector xN$.  Then since $A_{N-1}$ is regular there exists $h\in J_{\vector y N} (A_{N-1})$ with $h(\vector xN)=1$.  Regarding $h$ as a function on $X_N$ independent of the last coordinate yields the desired function.

Now suppose that $\vector y N = \vector xN$.  Let $W_1$ be the open disc of radius $1$ in $\C$ centered at $x_{N+1}$, and let $W_2$ be the union of the open discs of radius $1$ centered at each of the points $x_{N+1} + 2\pi ki$, $k\in \Z\sm\{0\}$ such that $(\row xN, x_{N+1} + 2\pi ki)$ is in $X_N$.  
Denote the function $X_N\ra \C$ sending each point of $X_N$ to its $(N+1)$-th coordinate by $\tau$.
Let 
$$W=\tau^{-1}(W_1 \cup W_2)$$
and 
$$C_1=\tau^{-1}(\ol W_1). \phantom{\cup W_2}$$
Then $X_N\sm W$ is a compact set disjoint from the fiber $\pi_N^{-1}(\vector xN)$, where $\pi_N: X_{N+1}\ra X_N$ is the canonical projection.  Therefore,
it follows from the previous paragraph that there exists a function $g$ in $A_N$ such that $g(x)=1$ and $g$ is identically zero on $X_N\sm W$.  Also, by Runge's theorem, there exists a sequence of polynomials $(q_k)$ satisfying
\begin{itemize}
\item[(a)] $q_k\ra 1$ uniformly on $\ol W_1$.
\item[(b)] $q_k\ra 0$ uniformly on $\ol W_2$.
\end{itemize}
Set $h_k=(q_k\circ \tau) g$ for each $k\in \Z_+$.  Then $h_k$ is in $A_N$, and $h_k\ra h$ uniformly on $X_N$, where
$$
h(z) = 
\begin{cases} 0 &\mbox{for\ } z\in X_N\sm C_1 \\
g(z) & \mbox{for\ } z\in C_1.
\end{cases}
$$
The function $h$ is in $A_N$, is zero on the neighborhood $X_N\sm C_1$ of $y$, and satisfies $h(x)=1$. This completes the inductive proof that $A_N$ is regular for all finite $N$.  

To show that $\Ao$ is regular it suffices to show that given distinct points $x=\sequence xn$ and $y=\sequence y n$ in $\Xo$ there exists a function $h\in J_y(\Ao)$ with $h(x)=1$.  Choose $N\in \Z_+\cup\{0\}$ such that $x_{N+1}\neq y_{N+1}$.  By the regularity of $A_N$, there exists $h\in J_{\vector y {N+1}}(A_N)$ with $h\vector x{N+1} = 1$.  When regarded as a function on $\Xo$ independent of the extra variables, $h$ has the required properties to establish the regularity of $\Ao$.

\medskip
\noindent{\emph{Step 11: We prove that $A_N$ is strongly regular for every $1\leq N\leq \omega$.}}

We first show that $A_N$ is strongly regular for each finite $N$.  Fix $N$ finite and $x\in X_N$ arbitrary.  We know from the previous step that $A_N$ is regular on $X_N$.  Since $X_N$ is rationally convex, it follows by a standard result \cite[Theorem~27.2]{S1} that $X_N$ is normal.  Therefore, it suffices, by the localness of strong regularity for normal uniform algebras \cite[Theorem~4.1]{Izzo2}, to show that there is a closed neighborhood $L$ of $x$ such that $\ol {A_N|L}$ is strongly regular at $x$.  Note that $z_1$ is locally a biholomorphic map from $\Sigma_N$ to $\C$.  Choose a closed disc $\ol\Delta$ in $\C$ centered at $z_1(x)$ such that $z_1$ takes a closed neighborhood $L$ of $x$ in $\Sigma_N$ homeomorphically onto $\ol\Delta$, and $z_1$ takes a neighborhood of $L$ biholomorphically onto a neighborhood of $\ol\Delta$.  Fix $f\in M_x(\ol{A_N|(X_N\cap L)})$ and $\vep>0$ arbitrary.  Choose a function $h\in \sO(X_N)$ such that $h(x)=0$ and $\|f - h\|_{X_N\cap L} < \vep/2$.  The function $z_1$ takes $X_N\cap L$ bijectively onto $K\cap\ol\Delta$.  Let $h^*$ be the holomorphic function defined on a neighborhood of $K\cap \ol\Delta$ such that $h^*\circ z_1=h$.  Then $h^*$ is in $R(K\cap\ol\Delta)=\ol{R(K)|(K\cap\ol\Delta)}$, and $h^*\bigl(z_1(x)\bigr)=0$.  By Lemma~\ref{straightforward}, $\ol{R(K)|(K\cap\ol\Delta)}$, is strongly regular, so there is a function $g^*$ in $\ol{R(K)|(K\cap\ol\Delta)}$, such that $\| h^*-g^* \|_{K\cap \ol\Delta} < \vep/2$ and $g^*=0$ on a neighborhood of $z_1(x)$ in $K\cap \ol\Delta$.  Set $g=g^*\circ z_1$.  Then $g$ is in $\ol{A_N|(X_N\cap L)}$, and $g=0$ on a neighborhood of $x$ in $X_N\cap L$.  Furthermore
$$\| h - g \|_{(X_N\cap L)} = \| h^* -g^* \|_{(K\cap \ol\Delta)} < \vep/2,$$
so $\| f - g \|_{(X_N\cap L)}<\vep$.  Thus $\ol{A_N|(X_N\cap L)}$ is strongly regular at $x$.  This completes the proof that $A_N$ is strongly regular.

To prove now that $\Ao$ is strongly regular, fix $x\in \Xo$, fix $f\in M_x(\Ao)$, and fix $\vep>0$ arbitrary.  Since $\bigcup_{N=1}^\infty A_N$ is dense in $\Ao$, there exists, for some finite $N$, a function $h \in M_{\tilde x}(A_N)$ such that $\|f-h\|<\vep/2$, where $\tilde x$ is the image of $x$ under the projection of $\Xo$ onto $X_N$.  Since $A_N$ is strongly regular, there exists $g\in J_{\tilde x}(A_N)$ such that $\| h - g \|<\vep/2$.  Then, regarded as an element of $\Ao$, the function $g$ lies in $J_x(\Ao)$ and $\| f - g \| <\vep$.  Therefore, $\Ao$ is strongly regular.

\medskip
\noindent{\emph{Step: 12: We complete the proof.}}

It suffices to show that $\Ao$ is generated by two functions, for if $a_1$ and $a_2$ are generators for $\Ao$, and we set $X=\bigl\{ \bigl(a_1(x), a_2(x)\bigr) :\nobreak x\in\nobreak\Xo \bigr\}\subset\nobreak\C^2$, then $P(X)$ is isomorphic as a uniform algebra to $\Ao$, the space $X$ is homeomorphic to $\Xo$, and $X$ is polynomially convex.

Note that $\Ao=P(\Xo)$ because the set $\sequencecurly fn$ is dense in $\Ao$ by the argument in Step~4 and each $f_n$ lies in $P(\Xo)$ since $f_n=\exp(z_{n+1})$ on $\Xo$.  For each finite $N$, let $p_N$ and $q_N$ be polynomials on $\C^N$ such that $f_N=p_N/q_N$.  Then the functions $z_1, q_1^{-1}, z_2, q_2^{-1}, z_3, q_3^{-1}, \ldots$ generate $\Ao$.  Since each $p_n$ and $q_n$ is a polynomial in $\row zn$, and each $z_n$ for $n>1$ is a logarithm of $p_{n-1}/q_{n-1}$, Lemma~\ref{generators} shows that $\Ao$ is generated by 2 functions.

%
%
%
%

\section{Proof of Theorem~\ref{three-ball}}\label{three-ball-section}

In this section we prove Theorem~\ref{three-ball} using the general method for constructing essential uniform algebras given in \cite{FI}.  We first recall the relevant parts of the main theorem of \cite{FI}.

\bthm\label{general-method}
Let $A$ be a nontrivial uniform algebra on a compact space $K$.  Let $X$ be a compact metric space every
{non-empty} open subset of which contains a nowhere dense subspace homeomorphic to $K$.  
Then there exists a sequence $\{K_n\}_{n=1}^\infty$ of pairwise disjoint, nowhere dense subspaces of $X$ each homeomorphic to $K$ such that $\bigcup_{n=1}^\infty K_n$ is dense in $X$ and ${\rm diam} (K_n)\rightarrow 0$.  If  homeomorphisms $h_n:K_n\rightarrow K$ are chosen and we set $A_n=\{f\circ h_n:f\in A\}$, then the collection of functions $\tA=\{f\in C(X): f|K_n\in A_n \hbox{\ for\ all\ $n$}\}$ is an essential uniform algebra on $X$.
The equality $\tA|K_n=A_n$ holds for all $n$.

Furthermore, the following relations hold between the properties of $A$ and the properties of $\tA$:
\item(i) $\tA$ is natural if and only if $A$ is natural;
\item(ii) $\tA$ is strongly regular if and only if $A$ is strongly regular.
\ethm

Several additional relations between the properties of $A$ and the properties of $\tA$ are proved in \cite{FI} but need not concern us here.  We will, however, need an addition relation not considered in \cite{FI}.

\bthm\label{essential-Dirichlet}
In the context of Theorem~\ref{general-method}, $\tA$ is Dirichlet on $X$ if and only if $A$ is Dirichlet on $K$.
\ethm

Before proving Theorem~\ref{essential-Dirichlet}, we use it to establish Theorem~\ref{three-ball}.

\bpf[Proof of Theorem~\ref{three-ball}]
Let $K$ be the set $X$ of Theorem~\ref{Dirichlet-PX}, let $A=P(K)$, and let $X=M$.  Since $K$ is a compact metrizable space of topological dimension~$1$, the hypotheses of Theorem~\ref{general-method} are satisfied (by \cite[Theorems IV 3 and V 2]{HW}).  Thus, by Theorems~\ref{general-method} and~\ref{essential-Dirichlet}, the uniform algebra $\tA$ furnished by Theorem~\ref{general-method} has all the properties listed in the statement of the theorem being proven.
\epf

The proof of Theorem~\ref{essential-Dirichlet} uses a lemma.

\blem\label{S-W}
In the context of Theorem~\ref{general-method}, the set 
$$G=\{ f\in \crx X : f|K_n {\rm \ is\  constant\ for\ all\ but\ finitely\ many\ } n\}$$
is dense in $\crx X$.
\elem

\bpf
Obviously $G$ is a subalgebra of $\crx X$ that contains the constant functions.  Therefore, by the Stone-Weierstrass theorem it suffices to show that $G$ separates points.  To that end, let $a$ and $b$ be two distinct points of $X$.

Let $Y$ be the quotient space obtained from $X$ by identifying each $K_n$ to a point, and let $q: X \ra Y$ be the quotient map.  Also, for each positive integer $m$, let $Y_m$ be the quotient space obtained from $X$ by identifying each $K_n$ with $n\neq m$ to a point, and let $q_m: X\ra Y_m$ be the quotient map.  Each of the compact spaces $Y$ and $Y_m$ is metrizable (by \cite[Lemma~2.6]{FI} for instance).

If there is no positive integer $m$ such that both $a$ and $b$ belong to $K_m$, then $q(a)\neq q(b)$, so there is a continuous real-valued function $f$ on $Y$ that separates $q(a)$ and $q(b)$.  Then $f\circ q$ is in $G$ and separates $a$ and $b$.  Suppose now that for some positive integer $m$, both $a$ and $b$ belong to $K_m$.  There is a continuous real-valued function $g$ on $q_m(K_m)$ that separates $q_m(a)$ and $q_m(b)$.  Apply the Tietze extension theorem to extend $g$ to a continuous real-valued function $f$ on $Y_m$.  Then $f\circ q_m$ is in $G$ and separates $a$ and $b$.
\epf

\bpf[Proof of Theorem~\ref{essential-Dirichlet}]
That $A$ is Dirichlet on $K$ whenever $\tA$ is Dirichlet on $X$ is essentially trivial and hence left to the reader.

We suppose now that $A$ is Dirichlet on $K$ and will prove that then $\tA$ is Dirichlet on $X$.  Let $f$ in $\crx X$ and $\vep>0$ be arbitrary.  By Lemma~\ref{S-W}, there is a positive integer $N$ and a function $g$ in $\crx X$ such that $g|K_n$ is constant for all $n>N$ and $\|g-f\|_X< \vep/2$.  Since $A$ is Dirichlet on $K$, there is, for each $n=1,\ldots, N$, a function $\alpha_n$ in $\crx {K_n}$ with $\|\alpha_n\|_{K_n}<\vep/2$ such that $g|K_n - \alpha_n$ is the real part of some function $h_n$ in $A_n$.

Let $Z$ denote the quotient space obtained by identifying each $K_n$ for $n>N$ to a point, and let $p:X\ra Z$ be the quotient map.  The compact space $Z$ is metrizable \cite[Lemma~2.6]{FI}.

Let $g^*$ be the function in $\crx Z$ such that $g=g^*\circ p$.  For each $n=1,\ldots, N$, let $h_n^*$ and $\alpha_n^*$ be the functions in $C\bigl(p(K_n)\bigr)$ and $C_\R\bigl(p(K_n)\bigr)$, respectively, such that $h_n=h_n^*\circ(p|K_n)$ and $\alpha_n=\alpha_n^*\circ (p|K_n)$.  By the Tietze extension theorem, there is a function $\alpha^*$ in $\crx Z$ with $\|\alpha^*\|_Z<\vep/2$ such that $\alpha^*|\bigl(p(K_n)\bigr)=\alpha_n^*$ for every $n=1,\ldots, N$.  The Tietze extension theorem also yields a function $s$ in $\crx Z$ such that $s|\bigl(p(K_n)\bigr) = \Im h_n^*$.

Set $h=\bigl[ (g^* - \alpha^*) + is\bigr] \circ p$.  On each $K_n$, for $n=1,\ldots, N$, the function $h$ coincides with $h_n$.  On each $K_n$, for $n>N$, of course $h$ is constant.  Thus $h$ belongs to $\tA$.  Furthermore, 
$$\| \Re h - f\|_X \leq \| \Re h - g\|_X + \|g- f\|_X = \|\alpha^*\|_Z + \|g-f\|_X < \vep.$$
\epf

%
%
%
%

\section{A nonembedding result}

\begin{proposition}
There exists a simply coconnected, compact metrizable space of topological dimension~$1$ that does not embed in $\R^2$.
\end{proposition}

\bpf
A triode is a space homeomorphic to the letter \textsf{T}, i.e., to the set $\bigl([-1, 1] \times \{1\}\bigr) \cup \bigl(\{0\} \times [0, 1]\bigr)$ in the plane.  A theorem of R. L. Moore \cite{Moore} asserts that any collection of disjoint triodes in $\R^2$ is countable.

Let $C$ be a Cantor set, and let $T$ be a triode.  Then the compact metrizable space $C\times T$ has topological dimension~$1$ by a standard result on the dimension of a product space \cite[p.~33, Corollary]{HW}.  Also $\check H^1(C\times T; \Z)\approx \check H^1(C;\Z)=0$ since $C$ is a deformation retract of $C\times T$.  But $C\times T$ is the union of an uncountable collection of disjoint triodes and hence does not embed in $\R^2$ by Moore's theorem.
\epf

\end{document}